\documentclass[12pt]{article}
\usepackage{txfonts}
\usepackage{graphicx}
\usepackage{tikz}

\def\T{\Gamma} \def\D{\Delta} \def\Th{\Theta}
\def\Ld{\Lambda} \def\E{\Sigma} \def\O{\Omega}
\def\a{\alpha} \def\b{\beta} \def\g{\gamma} \def\d{\delta} \def\e{\varepsilon}
\def\r{\rho} \def\o{\sigma} \def\t{\tau} \def\w{\omega} \def\k{\kappa}
\def\th{\theta} \def\ld{\lambda} \def\ph{\varphi} \def\z{\zeta}
\def\A{$A$~} \def\G{$G$~} \def\H{$H$~} \def\K{$K$~} \def\M{$M$~} \def\N{$N$~}
\def\P{$P$~} \def\Q{$Q$~} \def\R{$R$~} \def\V{$V$~} \def\X{$X$~} \def\Y{$Y$~}
\def\rmA{{\bf A}} \def\rmD{{\bf D}} \def\rmS{{\bf S}} \def\rmK{{\bf K}}
\def\rmM{{\bf M}} \def\Z{{\Bbb Z}} \def\GL{{\bf GL}} \def\C{Cayley}
\def\oa{\ovl A} \def\og{\ovl G} \def\oh{\ovl H} \def\ob{\ovl B} \def\oq{\ovl Q}
\def\oc{\ovl C} \def\ok{\ovl K} \def\ol{\ovl L} \def\om{\ovl M} \def\on{\ovl N}
\def\op{\ovl P} \def\oR{\ovl R} \def\os{\ovl S} \def\ot{\ovl T} \def\ou{\ovl U}
\def\ov{\ovl V} \def\ow{\ovl W} \def\ox{\ovl X} \def\oT{\ovl\T}
\def\lg{\langle} \def\rg{\rangle}
\def\di{\bigm|} \def\Di{\Bigm|} \def\nd{\mathrel{\bigm|\kern-.7em/}}
\def\Nd{\mathrel{\not\,\Bigm|}} \def\edi{\bigm|\bigm|}
\def\m{\medskip} \def\l{\noindent} \def\x{$\!\,$}  \def\J{$-\!\,$}
\def\Hom{\hbox{\rm Hom}} \def\Aut{\hbox{\rm Aut}} \def\Inn{\hbox{\rm Inn}}
\def\Syl{\hbox{\rm Syl}} \def\Sym{\hbox{\rm Sym}} \def\Alt{\hbox{\rm Alt}}
\def\Ker{\hbox{\rm Ker}} \def\fix{\hbox{\rm fix}} \def\mod{\hbox{\rm mod}}
\def\psl{{\bf P\!SL}} \def\Cay{\hbox{\rm Cay}} \def\Mult{\hbox{\rm Mult}}
\def\val{\hbox{\rm Val}} \def\Sab{\hbox{\rm Sab}} \def\supp{\hbox{\rm supp}}
\def\qed{\hfill $\Box$} \def\qqed{\qed\vspace{3truemm}}
\def\CS{\Cay(G,S)} \def\CT{\Cay(G,T)}
\def\h{\heiti\bf} \def\hs{\ziti{E}\bf} \def\st{\songti} \def\ft{\fangsong}
\def\kt{\kaishu} \def\heit{\hs\relax} \def\songt{\st\rm\relax}
\def\fangs{\ft\rm\relax} \def\kaish{\kt\rm\relax} \def\fs{\footnotesize}
\begin{document}

\newtheorem{theorem}{Theorem}[section]
\newtheorem{corollary}[theorem]{Corollary}
\newtheorem{definition}[theorem]{Definition}
\newtheorem{conjecture}[theorem]{Conjecture}
\newtheorem{question}[theorem]{Question}
\newtheorem{lemma}[theorem]{Lemma}
\newtheorem{proposition}[theorem]{Proposition}
\newtheorem{example}[theorem]{Example}
\newtheorem{problem}[theorem]{Problem}
\newenvironment{proof}{\noindent {\bf
Proof.}}{\rule{3mm}{3mm}\par\medskip}
\newcommand{\remark}{\medskip\par\noindent {\bf Remark.~~}}
\newcommand{\pp}{{\it p.}}
\newcommand{\de}{\em}

%=================================================
\def\dfrac{\displaystyle\frac} \def\ovl{\overline}
\def\for{\forall~} \def\exi{\exists~} \def\c{\subseteq}
\def\iif{\Longleftrightarrow} \def\Rto{\Rightarrow} \def\Lto{\Leftarrow}
%=================================================
\def\T{\Gamma} \def\D{\Delta} \def\Th{\Theta}
\def\Ld{\Lambda} \def\E{\Sigma} \def\O{\Omega}
\def\a{\alpha} \def\b{\beta} \def\g{\gamma} \def\d{\delta} \def\e{\varepsilon}
\def\r{\rho} \def\o{\sigma} \def\t{\tau} \def\w{\omega} \def\k{\kappa}
\def\th{\theta} \def\ld{\lambda} \def\ph{\varphi} \def\z{\zeta}
%=================================================
\def\A{$A$~} \def\G{$G$~} \def\H{$H$~} \def\K{$K$~} \def\M{$M$~} \def\N{$N$~}
\def\P{$P$~} \def\Q{$Q$~} \def\R{$R$~} \def\V{$V$~} \def\X{$X$~} \def\Y{$Y$~}
\def\rmA{{\bf A}} \def\rmD{{\bf D}} \def\rmS{{\bf S}} \def\rmK{{\bf K}}
\def\rmM{{\bf M}} \def\Z{{\Bbb Z}} \def\GL{{\bf GL}} \def\C{Cayley}
%=================================================
\def\oa{\ovl A} \def\og{\ovl G} \def\oh{\ovl H} \def\ob{\ovl B} \def\oq{\ovl Q}
\def\oc{\ovl C} \def\ok{\ovl K} \def\ol{\ovl L} \def\om{\ovl M} \def\on{\ovl N}
\def\op{\ovl P} \def\oR{\ovl R} \def\os{\ovl S} \def\ot{\ovl T} \def\ou{\ovl U}
\def\ov{\ovl V} \def\ow{\ovl W} \def\ox{\ovl X} \def\oT{\ovl\T}
\def\lg{\langle} \def\rg{\rangle}
%=================================================
\def\di{\bigm|} \def\Di{\Bigm|} \def\nd{\mathrel{\bigm|\kern-.7em/}}
\def\Nd{\mathrel{\not\,\Bigm|}} \def\edi{\bigm|\bigm|}
\def\m{\medskip} \def\l{\noindent} \def\x{$\!\,$}  \def\J{$-\!\,$}
%=================================================
\def\Hom{\hbox{\rm Hom}} \def\Aut{\hbox{\rm Aut}} \def\Inn{\hbox{\rm Inn}}
\def\Syl{\hbox{\rm Syl}} \def\Sym{\hbox{\rm Sym}} \def\Alt{\hbox{\rm Alt}}
\def\Ker{\hbox{\rm Ker}} \def\fix{\hbox{\rm fix}} \def\mod{\hbox{\rm mod}}
\def\psl{{\bf P\!SL}} \def\Cay{\hbox{\rm Cay}} \def\Mult{\hbox{\rm Mult}}
\def\val{\hbox{\rm Val}} \def\Sab{\hbox{\rm Sab}} \def\supp{\hbox{\rm supp}}
\def\qed{\hfill $\Box$} \def\qqed{\qed\vspace{3truemm}}
\def\CS{\Cay(G,S)} \def\CT{\Cay(G,T)}
%=================================================
\def\h{\heiti\bf} \def\hs{\ziti{E}\bf} \def\st{\songti} \def\ft{\fangsong}
\def\kt{\kaishu} \def\heit{\hs\relax} \def\songt{\st\rm\relax}
\def\fangs{\ft\rm\relax} \def\kaish{\kt\rm\relax} \def\fs{\footnotesize}
%=================================================

\newcommand{\JEC}{{\it Europ. J. Combinatorics},  }
\newcommand{\JCTB}{{\it J. Combin. Theory Ser. B.}, }
\newcommand{\JCT}{{\it J. Combin. Theory}, }
\newcommand{\JGT}{{\it J. Graph Theory}, }
\newcommand{\ComHung}{{\it Combinatorica}, }
\newcommand{\DM}{{\it Discrete Math.}, }
\newcommand{\ARS}{{\it Ars Combin.}, }
\newcommand{\SIAMDM}{{\it SIAM J. Discrete Math.}, }
\newcommand{\SIAMADM}{{\it SIAM J. Algebraic Discrete Methods}, }
\newcommand{\SIAMC}{{\it SIAM J. Comput.}, }
\newcommand{\ConAMS}{{\it Contemp. Math. AMS}, }
\newcommand{\TransAMS}{{\it Trans. Amer. Math. Soc.}, }
\newcommand{\AnDM}{{\it Ann. Discrete Math.}, }
\newcommand{\NBS}{{\it J. Res. Nat. Bur. Standards} {\rm B}, }
\newcommand{\ConNum}{{\it Congr. Numer.}, }
\newcommand{\CJM}{{\it Canad. J. Math.}, }
\newcommand{\JLMS}{{\it J. London Math. Soc.}, }
\newcommand{\PLMS}{{\it Proc. London Math. Soc.}, }
\newcommand{\PAMS}{{\it Proc. Amer. Math. Soc.}, }
\newcommand{\JCMCC}{{\it J. Combin. Math. Combin. Comput.}, }
\newcommand{\GC}{{\it Graphs Combin.}, }

\title{ The Number of Subtrees of Trees with Given Degree Sequence \thanks{ This work is supported by the National
Natural Science Foundation of China (No:10971137), the National
Basic Research Program (973) of China (No.2006CB805900),  and a
grant of Science and Technology Commission of Shanghai Municipality
(STCSM, No: 09XD1402500) .\newline \indent
 $^{\dagger}$Corresponding author: Xiao-Dong
Zhang (Email: xiaodong@sjtu.edu.cn)}}
\author{Xiu-Mei Zhang$^{1,2}$, Xiao-Dong Zhang$^{1,\dagger}$,
Daniel Gray$^{3}$ and Hua  Wang$^{4}$\\
{\small $^{1}$Department of Mathematics,
 Shanghai Jiao Tong University} \\
{\small  800 Dongchuan road, Shanghai, 200240, P. R. China}\\
{\small $^{2}$Department of Mathematics,Shanghai sandau University}\\
{\small  2727 jinhai road, Shanghai, 201209, P. R. China}\\
{\small $^{3}$Department of Mathematics, University of Florida}\\
{\small Gainesville, FL 32611, United States}\\
{\small $^{4}$Department of Mathematical Sciences, Georgia Southern
University} \\
{\small Stateboro, GA 30460,  United States}
 %{\small Emails: wfluckyzxm@163.com;  xiaodong@sjtu.edu.cn}\\
 }
\date{}
\maketitle
 \begin{abstract}
  This paper investigates some properties of
  the number of subtrees of a tree with given degree sequence. These
  results are used to characterize trees
  with the given degree sequence that have the largest number of
  subtrees, which generalizes the recent results of Kirk and Wang.
  These trees coincide with those which were proven by Wang and
  independently Zhang et al. to minimize the Wiener index. We also provide a partial ordering of the extremal trees with different degree sequences, some extremal results follow as corrollaries.
 \end{abstract}


 \vskip 0.5cm

\section{Introduction}
All graphs in this paper will be finite, simple and undirected. A
$tree$ $T=(V,E)$ is a connected, acyclic graph where $V(T)$ and $E(T)$ denote the
vertex set and edge set respectively. We refer to vertices of degree 1 of $T$ as
$leaves$. The unique path connecting two vertices $u,v$ in $T$ will
be denoted by $P_{T}(u,v)$.  The  number of edges on $P(u,v)$ is
called distance $dist_T(u,v)$, or for short $dist(u,v)$  between
them. We call a tree $(T,r)$ rooted at the vertex $r$ (or just by
$T$ if it is clear what the root is) by specifying a vertex $r\in
V(T)$. The $height$ of a vertex $v$ of a rooted tree $T$ with root
$r$ is $h_{T}(v)=dist_{T}(r,v)$. For any two different vertices
$u,v$ in a rooted tree $(T,r)$, we say that $v$ is a $successor$ of
$u$ and $u$ is an $ancestor$ of $v$ if $P_{T}(r,u)\subset
P_{T}(r,v)$. Furthermore, if $u$ and $v$ are adjacent to each other
and $dist_{T}(r,u)=dist_{T}(r,v)-1$, we say that $u$ is the $parent$
of $v$ and $v$ is a $child$ of $u$. Two vertices $u,v$ are siblings
of each other if they share the same parent. A subtree of a tree
will often be described by its vertex set.

The number of subtrees of a tree has received much attention. It is
well known that the path $P_n$ and the star $K_{1,n-1}$ have the
most and least subtrees among all trees of order $n$, respectively.
The binary trees that maximize or minimize the number of subtrees
are characterized in \cite{Szelely2005, Szelely2007}.

Formulas are given to calculate the number of subtrees of these
extremal binary trees. These formulas use a new representation of
integers as a sum of powers of 2. Number theorists have already
started investigating this new binary representation
\cite{Heuberger2007}. Also, the sequence of the number of subtrees
of these extremal binary trees (with $2l$ leaves, $l = 1, 2,
\cdots$) appears to be new \cite{Sequence}. Later, a linear-time
algorithm to count the subtrees of a tree is provided in
\cite{yan2006}.

In a related paper \cite{Szekely2005}, the number of leaf-containing
subtrees are studied for binary trees. The results turn out to be
useful in bounding the number of acceptable residue configurations.
See \cite{knudsen2003} for details.

An interesting fact is that among binary trees of the same size, the extremal one
that minimizes the number of subtrees is exactly the one that maximizes some chemical
indices such as the well known Wiener index, and vice versa.
In \cite{kirk2008}, subtrees of trees with given order and maximum vertex
degree are studied. The extremal trees coincide with
the ones for the Wiener index as well. Such correlations between different
topological indices of trees are studied in
\cite{wagner2007}.

Recently, in \cite{X.D.Zhang2008} and \cite{Hua2007} respectively, extremal
trees are characterized regarding the Wiener index with a given degree sequence.
Then it is natural to consider the following
question.

\begin{problem}
Given the degree sequence and the number of vertices of a tree, find the upper
bound for the number of subtrees, and characterize all
extremal trees that attain this bound.
\end{problem}

It will not be a surprise to see that such extremal trees coincide with the ones that attain the minimum Wiener index.
Along this line, we also provide an ordering of the degree sequences according to the largest number of subtrees.
With our main results, Theorems~\ref{theorem2-1} and \ref{theorem2-2}, one can  deduce
extremal graphs with the largest  number of subtrees  in some classes
of graphs. This generalizes the results of \cite{Szelely2005},
\cite{kirk2008}, etc.

The rest of this paper is
organized as follows: In Section 2, some notations and the main
theorems are stated. In Section 3, we present some observations regarding the structure of
 the extremal trees.
In Section 4, we present the proofs of the main theorems. In Section 5, we show,
as corollaries, characterizations of the extremal trees in different categories of
 trees including previously known results.

\section{Preliminaries}

For a nonincreasing sequence of positive integers $\pi=(d_0, \cdots,
d_{n-1})$ with $n\ge 3$, let ${\mathcal{T}}_{\pi}$ denote the set of
all trees with $\pi$ as its degree sequence. We can construct a
special tree $T_{\pi}^* \in \mathcal{T}_{\pi}$ by using
breadth-first search method as follows. Firstly, label the vertex
with the largest degree $d_0$ as $v_{01}$ (the root). Secondly,
label the neighbors of $v_0$ as $v_{11},v_{12},\dots,v_{1d_0}$ from
left to right and let $d(v_{1i})=d_i$ for $i=1, \cdots, d_0$. Then
repeat the second step for all newly labeled vertices until all
degrees are assigned. For example, if $\pi=(4, 4,3,
 3,3,3,3,2,1,1,1, 1, 1, 1,  1, 1, 1,1,1)$, $T^*_{\pi}$  is shown in Fig. 1.
  There is a vertex $v_{01}$ (the root) in layer 0 with the largest degree 4; its four
  neighbors are labeled as $v_{11}, v_{12}, v_{13}, v_{14}$ in layer 1, with
  degrees 4, 3, 3, 3 from left to right; nine
 vertices $v_{21},v_{22}, \cdots, v_{29}$ in layer 2; five
 vertices $v_{31}, v_{32}, v_{33}, v_{34}, v_{35}$ in layer 3. The number
 of vertices in each layer $i$, denoted by $s_i$ can be easily calculated as
 $s_0=1,$
 $ s_1=d_0=4,$ $s_2=d_1+d_2+d_3+d_4-s_1=4+3+3+3-4=9, $ and
 $s_3=d_5+\cdots+d_{13}-s_2=5$.

 \setlength{\unitlength}{0.1in}
\begin{picture}(60,30)
%\begin{figure}[ht]
%\begin{center}
%\begin{tikzpicture}
\put(28,27){\circle{0.5}}
 \put(28.7,27){$v_{01}$}
%\put(52,27){layer 0}
 \put(27.8,26.8){\line(-3,-1){20}}
\put(27.8,26.8){\line(-1,-1){6.6}}
 \put(28.2,26.8){\line(1,-1){6.6}}
\put(28.2,26.8){\line(3,-1){20}}

\put(7.6,20){\circle{0.5}} \put(20.9,20){\circle{0.5}}
\put(35,20){\circle{0.5}} \put(48.4,20){\circle{0.5}}

\put(9,19.5){$v_{11}$}\put(22,19.5){$v_{12}$}\put(36,19.5){$v_{13}$}\put(49.4,19.5){$v_{14}$}

\put(7.4,19.8){\line(-1,-1){5.1}} \put(7.6,19.78){\line(0,-1){5}}
 \put(7.8,19.8){\line(1,-1){5.1}}
\put(2.2,14.5){\circle{0.5}} \put(7.6,14.5){\circle{0.5}}
\put(13,14.5){\circle{0.5}}

\put(20.7,19.8){\line(-1,-1){5.1}} \put(21.1,19.8){\line(1,-1){5.1}}
\put(15.4,14.5){\circle{0.5}} \put(26.4,14.5){\circle{0.5}}

\put(3.2,13.3){$v_{21}$} \put(8.5,13.3){$v_{22}$}
\put(11,13.3){$v_{23}$} \put(15.25,13.3){$v_{24}$}
\put(26.5,13.3){$v_{25}$} \put(29.5,13.3){$v_{26}$}
\put(40,13.3){$v_{27}$} \put(42.5,13.3){$v_{28}$}
\put(53,13.34){$v_{29}$}

 \put(34.8,19.8){\line(-1,-1){5.1}}
\put(35.2,19.8){\line(1,-1){5.1}} \put(29.5,14.5){\circle{0.5}}
\put(40.4,14.5){\circle{0.5}}

\put(48.2,19.8){\line(-1,-1){5.1}} \put(48.6,19.8){\line(1,-1){5.1}}
 \put(43,14.5){\circle{0.5}}
\put(53.8,14.5){\circle{0.5}}

\put(2,14.3){\line(-1,-2){2}}
 \put(2.4,14.3){\line(1,-2){2}}
 \put(-0.15,10){\circle{0.5}}
\put(4.55,10){\circle{0.5}}

\put(7.6,14.3){\line(-1,-2){2}}
 \put(5.6,10){\circle{0.5}}
 \put(-0.4,8.7){$v_{31}$} \put(3.1,8.7){$v_{32}$} \put(6, 8.7){$v_{33}$}
\put(9.5,8.7){$v_{34}$} \put(13.6,8.7){$v_{35}$}

\put(9.5,10){\circle{0.5}} \put(7.6,14.3){\line(1,-2){2}}
\put(13,10){\circle{0.5}} \put(13,14.25){\line(0,-1){3.8}}

 \put(25,5){\bf Figure 1}
 % \end{tikzpicture}
%\end{center}
%\end{figure}
          \end{picture}

To explain the structure and properties of
$T_{\pi}^*$, we need the following notation from \cite{zhang2008}.
\begin{definition}(\cite{zhang2008})\label{def2-1}
Let $T=(V,E)$ be a tree with root $v_0$. A well-ordering $\prec$ of
the vertices is called a BFS-ordering  if $\prec$ satisfies the
following properties.

(1) If $u, v\in V$, and $u\prec v$, then $h(u)\le h(v)$ and $d(u)\ge
d(v)$;

(2) If there are two edges $ uu_1\in E(T)$ and
 $vv_1\in E(T)$ such that $u\prec v$, $h(u)=h(u_1)-1$ and $ h(v)=h(v_1)-1$,
 then $u_1\prec v_1$.

 We call trees that have a BFS-ordering of its vertices a
 BFS-tree.
\end{definition}
It is easy to see that  $T_{\pi}^*$ has a  BFS-ordering and
any two BFS-trees with degree sequence $\pi$
are isomorphic (for example, see \cite{zhang2008}). And the BFS-trees are extremal with respect to the Laplacian spectral radius.%最后一句是我根据审稿人意见加上去的。

Let $\pi=(d_0,\cdots,
d_{n-1})$ and
 $\pi^{\prime}=(d_0^{\prime}, \cdots, d_{n-1}^{\prime})$ be two
 nonincreasing sequences. If
 $\sum_{i=0}^{k}d_i\le\sum_{i=0}^kd_i^{\prime}$ for $k=0, \cdots, n-2$ and
 $\sum_{i=0}^{n-1}d_i=\sum_{i=0}^{n-1}d_i^{\prime}$, then the
 sequence $\pi^{\prime}$ is said to {\it major} the sequence $\pi$ and denoted by
 $\pi\triangleleft \pi^{\prime}$. It is known that the following
 holds (for example, see \cite{wei1982} or \cite{zhang2008}).
  \begin{proposition}(Wei \cite{wei1982})
 \label{prop}
 Let $\pi=(d_0, \cdots d_{n-1})$ and $\pi^{\prime}=(d_0^{\prime}, \cdots,
 d_{n-1}^{\prime})$ be two  nonincreasing graphic  degree sequences. If
 $\pi\triangleleft \pi^{\prime},$ then there exists a series of
 graphic degree sequences  $\pi_1, \cdots, \pi_k$ such that
 $\pi \triangleleft \pi_1\triangleleft \cdots \triangleleft
\pi_k\triangleleft \pi^{\prime}$, where $\pi_i$ and $\pi_{i+1}$
differ at exactly two entries, say $d_j$ ($d'_j$) and $d_k$ ($d'_k$)
of $\pi_i$ ($\pi_{i+1}$), with $d'_j = d_j +1$, $d'_k = d_k -1$ and
$j<k$.
  \end{proposition}

  The main results of this paper can be stated as follows.
  \begin{theorem}\label{theorem2-1}
   With a given degree sequence $\pi$, $T_{\pi}^*$ is the unique
   tree with the largest number of subtrees in ${\mathcal{T}}_{\pi}$.
  \end{theorem}

  \begin{theorem}\label{theorem2-2}
  Given two different degree sequences $\pi$ and $\pi_1$. If  $\pi \triangleleft
  \pi_1$, then the number of subtrees of $T_{\pi}^*$ is less than
  the number of subtrees of $T_{\pi_1}$.
  \end{theorem}

%\noindent  {\bf Remark} Theorem~\ref{theorem2-2} may be considered as the majorization
  %theorem on the degree sequences.

\section{Some Observations}

In order to prove Theorems~\ref{theorem2-1} and \ref{theorem2-2}, we
need to introduce some more terminologies. For a vertex $v$ of a
rooted tree $(T,r)$, let $T(v)$, the subtree induced by $v$, denote
the subtree of $T$ (rooted at $v$) that is induced by $v$ and all
its successors. For a tree $T$ and vertices
$v_{1}$,$v_{2}$,\dots,$v_{m-1}$,$v_{m}$ of $T$, let
$f_{T}(v_{1}$,$v_{2}$,\dots,$v_{m-1}$,$v_{m}$) denote the number of
subtrees of T that contain the vertices
$v_{1}$,$v_{2}$,$\dots$,$v_{m-1}$,$v_{m}$. In particular, $f_{T}(v)$
denotes the number of subtrees of $T$ that contain $v$. Let
$\varphi(T)$ denote the number of non-empty subtrees of $T$.

Let $W$ be a tree and $x, y$ be two vertices of $W$. The path
$P_{W}(x, y)$ from $x$ to $y$ can be denoted by $x_{m}x_{m-1}\dots
x_{2}x_{1}y_{1}y_{2}\dots y_{m-1}y_m$ for odd $dist(x,
y)$ or $ x_{m}x_{m-1}\dots x_{2}x_{1}zy_{1}y_{2}\dots y_{m-1}y_{m}$
for even $dist(x,y)$, where $x_m\equiv x, y_m\equiv y$. Let
$G_1$ be the graph resulted from $W$ by deleting all edges in
$P_W(x,y)$. The connected components (in $G_1$) containing
$x_{i}$, $y_{i}$ and $z$
 are denoted by $X_{i}$, $Y_i$ and $Z$, respectively, for $i=1, 2,\dots, m$.
 We also let  $X_{\geq k}$ be the connected component of $W$ containing $x_k$
 after deleting the edge
  $x_{k-1}x_k$  and  $Y_{\geq k}$ be the connected component
  of $W$ containing $y_k$ after deleting the edge $y_{k-1}y_k$, for $k=1, \cdots, m$.
   Figure~\ref{fig-path} shows such a labelling according to a path of odd length
   (without $z$).

\begin{figure}[ht]
\begin{center}
\begin{tikzpicture}

\draw(-1,0)--(0.5,0);
 \draw(1.5,0)--(5.5,0); \draw(6.5,0)--(8,0);

\draw [fill=blue!0](0,0)--(-0.25,0.5)--(0.25,0.5)--(0,0);
\draw[fill=blue!0](2,0)--(1.75,0.5)--(2.25,0.5)--(2,0);
\draw[fill=blue!0](3,0)--(2.75,0.5)--(3.25,0.5)--(3,0);
\draw[fill=blue!0](4,0)--(3.75,0.5)--(4.25,0.5)--(4,0);
\draw[fill=blue!0](5,0)--(4.75,0.5)--(5.25,0.5)--(5,0);
\draw[fill=blue!0](7,0)--(6.75,0.5)--(7.25,0.5)--(7,0);
\draw[fill=blue!0](-1,0)--(-3,1)--(-3,-1)--(-1,0);
 \draw[fill=blue!0] (8,0)--(10,1)--(10,-1)--(8,0);

\node at(0,0.75)[]{$X_{k-1}$};
\node at(1,0)[]{$\ldots$};
 \node at(2,0.75)[]{$X_2$};
  \node at(3,0.75) [] {$X_1$}; \node at (4,0.75)
[] {$Y_1$}; \node at (5,0.75) [] {$Y_2$}; \node at (6,0) []
{$\ldots$}; \node at (7,0.75) [] {$Y_{k-1}$}; \node at (-2,0) []
{$X_{\geq k}$}; \node at (9, 0) [] {$Y_{\geq k}$};

\node at (-1,0)[draw, circle, inner sep=0.5mm, label=below:$x_k$,
fill=green!0] {};
 \node at (0,0)[draw, circle, inner sep=0.5mm,
label=below:$x_{k-1}$, fill=green!0] {};
 \node at (2,0)[draw,
circle, inner sep=0.5mm, label=below:$x_2$, fill=green!0]{};
\node
at (3,0) [draw, circle, inner sep=0.5mm, label=below:$x_1$,
fill=green!0] {}; \node at (4,0) [draw, circle, inner sep=0.5mm,
label=below:$y_1$, fill=green!0] {}; \node at (5,0) [draw, circle,
inner sep=0.5mm, label=below:$y_2$, fill=green!0] {}; \node at (7,0)
[draw, circle, inner sep=0.5mm, label=below:$y_{k-1}$, fill=green!0]
{}; \node at (8,0) [draw, circle, inner sep=0.5mm,
label=below:$y_k$, fill=green!0] {};

\end{tikzpicture}
\end{center}
\addtocounter{figure}{1}
 \caption{Labelling of a path and the
components} \label{fig-path}
\end{figure}
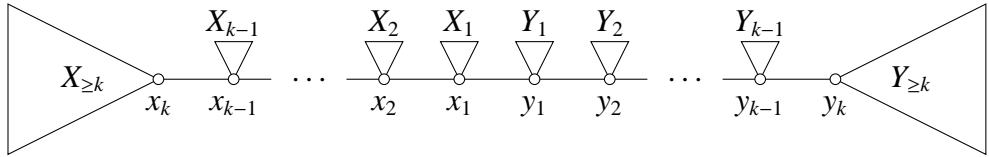

We need the next two lemmas from \cite{kirk2008} to proceed.

\begin{lemma}(\cite{kirk2008})\label{lemma3-1}
Let $W$ be a tree with a path $P_W(x_m,y_m)=x_{m}x_{m-1}\dots
x_{2}x_{1}(z)y_{1}y_{2}\dots y_{m-1}y_m$  from $x_m$ to $y_m$. If
  $f_{X_{i}} (x_{i})\geq f_{Y_{i}} (y_{i})$ for $i =1, 2,\dots,m$,
then $f_{W}(x_m)\geq f_{W} (y_m)$. Furthermore, if this inequality holds, then$f_{W}(x_m)= f_{W}
(y_m)$ if and only if  $f_{X_{i}} (x_{i})=f_{Y_{i}} (y_{i})$ for $i
=1, 2,\dots,m$.
\end{lemma}

Now let $X$ and $Y$
be two rooted trees  with roots $x'$ and $y'$. Let $T$ be a
 tree containing vertices $x$ and $y$. Then we can build
$T^{\prime}$ by identifying the root $x'$ of $X$ with $x$ of $T$ and the root
$y'$ of $Y$ with $y$ of $T$, and $T^{\prime\prime}$ by
identifying the root $x'$ of $X$ with $y$ of
$T$ and the root $y'$ of $Y$ with $x$ of $T$.

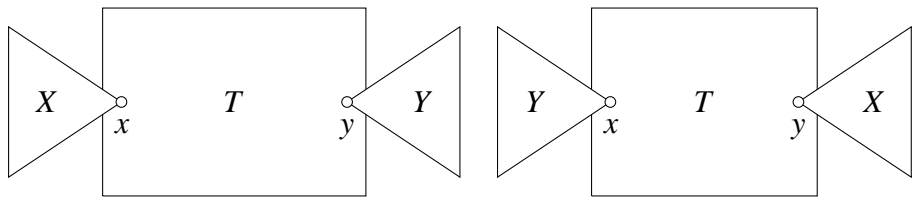
\begin{figure}[ht]
\begin{center}
\begin{tikzpicture}

\draw[fill=blue!0](-1.25,-1.25)--(-1.25,1.25)--(2.25,1.25)--(2.25,-1.25)--(-1.25,-1.25);
\draw[fill=blue!0](-1,0)--(-2.5,1)--(-2.5,-1)--(-1,0);
 \draw[fill=blue!0] (2,0)--(3.5,1)--(3.5,-1)--(2,0);

\node at(0.5,0)[]{$T$};
 \node at(-2,0)[]{$X$};
 \node at(3, 0)[]{$Y$};

\node at(-1,0)[draw, circle, inner sep=0.5mm, label=below:$x$,
fill=green!0] {}; \node at(2,0) [draw, circle, inner sep=0.5mm,
label=below:$y$, fill=green!0]{};

\draw[fill=blue!0](5.25,-1.25)--(5.25,1.25)--(8.25,1.25)--(8.25,-1.25)--(5.25,-1.25);
\draw[fill=blue!0](5.5,0)--(4,1)--(4,-1)--(5.5,0);
\draw[fill=blue!0](8,0)--(9.5,1)--(9.5,-1)--(8,0);

\node at (6.75,0)[]{$T$};
 \node at (4.5,0)[] {$Y$}; \node at(9,0) []
{$X$};

\node at (5.5,0) [draw, circle, inner sep=0.5mm, label=below:$x$, fill=green!0] {};
\node at (8,0) [draw, circle, inner sep=0.5mm, label=below:$y$, fill=green!0] {};

\end{tikzpicture}
\end{center}
\caption{Constructing $T'$ (left) and $T''$ (right)}
\label{fig-2T}
\end{figure}

\begin{lemma}(\cite{kirk2008})\label{lemma3-2}
Let $T$ ,$T^{\prime}$,$T^{\prime\prime}$ be as in Figure 2.
%Let $T$ be a tree with two vertices $x,y$.
If $f_{}(x)\ge f_{W}
(y)$ and $f_{X}(x)\le f_{Y} (y)$, then $\varphi(T^{\prime\prime})\ge
\varphi(T^{\prime})$ with equality if and only if $f_{T}(x)=f_{T}
(y)$ or
 $f_{X}(x')= f_{Y} (y')$.
\end{lemma}

From Lemmas \ref{lemma3-1} and \ref{lemma3-2},
we immediately achieve
the following observation. We leave the proof to the reader.

\begin{lemma}\label{lemma3.3}
Let $T$ be a  tree in $\mathcal{T}_{\pi}$ and
$P(x_m,y_m)=x_{m}x_{m-1}\dots x_{2}x_{1}(z)y_{1}y_{2}\dots
y_{m-1}y_m $ be a path of $T$. Let $T^{\prime}$ be the tree from $T$
by deleting the  two edges $x_{k}x_{k+1}$ and $y_{k}y_{k+1}$ and adding
two edges $x_{k+1}y_k$ and $y_{k+1}x_k$. If $f_{X_{i}}(x_i)\ge
f_{Y_{i}}(y_{i})$ for $i=1, \cdots, k$ and $1\le k\le m-1$, and
$f_{X_{\ge k+1}}(x_{k+1})\le f_{Y_{\ge k+1}}(y_{k+1}),$ then
$$\varphi(T)\le \varphi (T^{\prime})$$
with equality if and only if $f_{X_{\ge k+1}}(x_{k+1})= f_{Y_{\ge
k+1}}(y_{k+1})$ or  $f_{X_{i}}(x_i)= f_{Y_{i}}(y_{i})$ for $i=1,
\cdots, k$.
\end{lemma}

For convenience, we refer to trees that maximize the number of subtrees as {\it optimal}.
In terms of the structure of the optimal tree, we have the following version of Lemma~\ref{lemma3.3}.

\begin{corollary}\label{corollary3.4}
Let $T$ be an optimal  tree in $\mathcal{T}_{\pi}$ and
$P(x_m,y_m)=x_{m}x_{m-1}\dots x_{2}x_{1}(z)y_{1}y_{2}\dots
y_{m-1}y_m $ be a path of $T$.  If $f_{X_{i}}(x_i)\ge
f_{Y_{i}}(y_{i})$ for $i=1, \cdots, k$  with at least one strict
inequality and $1\le k\le m-1$,   then $f_{X_{\ge k+1}}(x_{k+1})\ge
f_{Y_{\ge k+1}}(y_{k+1})$.
\end{corollary}

\begin{lemma}\label{lemma3.4}
Let $T$ be an optimal tree in $\mathcal{T}_{\pi}$ and
$P(x_m,y_m)=x_{m}x_{m-1}\dots x_{2}x_{1}(z)y_{1}y_{2}\dots
y_{m-1}y_m $ be a path of $T$. If $f_{X_{i}}(x_i)\ge
f_{Y_{i}}(y_{i})$ for $i=1, \cdots, k$ with at least one strict
inequality and $1\le k\le m-1$,
 then $f_{X_{k+1}}(x_{k+1})\ge
f_{Y_{k+1}}(y_{k+1})$ .
\end{lemma}
\begin{proof} If $k=m-1$, then by Corollary~\ref{corollary3.4}, the
assertion holds since $f_{X\ge m}(x_m)=f_{X_m}(x_m)$ and $f_{Y\ge
m}(y_m)=f_{Y_m}(y_m)$. Hence we assume that $1\le k\le m-2$. Suppose
that $f_{X_{k+1}}(x_{k+1})< f_{Y_{k+1}}(y_{k+1})$. Denote by $M$ the
number of  subtrees of $T$ not containing vertices $x_k$ and $y_k$.
Let $W$ be the connected component of $T$ by deleting the two edges
$x_kx_{k+1}$ and $y_ky_{k+1}$ containing vertices $x_k$ and $y_k$.
Then
\begin{eqnarray*}
\varphi(T)&=& \left\{1+f_{X_{k+1}}(x_{k+1})[1+f_{X_{\ge k+2}}(x_{k+2})]\right\}
[f_W(x_k)-f_W(x_k, y_k)]+\\
    & &\left\{1+f_{Y_{k+1}}(y_{k+1})[1+f_{Y_{\ge k+2}}(y_{k+2})]\right\}[f_W(y_k)-f_W(x_k,
    y_k)]+\\
&&\left\{1+f_{X_{k+1}}(x_{k+1})[1+f_{X_{\ge
k+2}}(x_{k+2})]\right\}\left\{1+f_{Y_{k+1}}(y_{k+1})[1+f_{Y_{\ge
k+2}}(y_{k+2})]\right\}f_W(x_k, y_k)+M.
\end{eqnarray*}
Oh the other hand, let $T^{\prime}$ be  the tree from $T$ by
deleting four edges $x_kx_{k+1}$, $x_{k+1}x_{k+2}$, $ y_ky_{k+1}$
and $y_{k+1}y_{k+2}$ and adding  four edges $x_ky_{k+1}$,
$y_{k+1}x_{k+2}$, $y_kx_{k+1}$ and $x_{k+1}y_{k+2}$. Clearly,
$T^{\prime}\in \mathcal{T}_{\pi}$ and
\begin{eqnarray*}
\varphi(T^{\prime})&=&\left\{1+f_{Y_{k+1}}(y_{k+1})[1+f_{X_{\ge
k+2}}(x_{k+2})]\right\}
[f_W(x_k)-f_W(x_k, y_k)]+\\
    & &\left\{1+f_{X_{k+1}}(x_{k+1})[1+f_{Y_{\ge k+2}}(y_{k+2})]\right\}[f_W(y_k)-f_W(x_k,
    y_k)]+\\
&&\left\{(1+f_{Y_{k+1}}(y_{k+1})[1+f_{X_{\ge
k+2}}(x_{k+2})]\right\}\left\{1+f_{X_{k+1}}(x_{k+1})[1+f_{Y_{\ge
k+2}}(y_{k+2})]\right\}f_W(x_k, y_k)+M.
\end{eqnarray*}
Hence
\begin{eqnarray*}
\varphi(T^{\prime})-\varphi(T)&=&(f_{Y_{k+1}}(y_{k+1})-f_{X_{k+1}}(x_{k+1}))\{[1+f_{X_{\ge
k+2}}(x_{k+2})][f_W(x_k)-f_W(x_k, y_k)]- \\
 &&[1+f_{Y_{\ge k+2}}(y_{k+2})][f_W(y_k)-f_W(x_k, y_k)]+f_w(x_k,y_k)(f_{X_{\ge
k+2}}(x_{k+2})-f_{Y_{\ge k+2}}(y_{k+2}))\}.
 %&& >0.
 \end{eqnarray*}
 Obviously, we have $f_W(y_k)>f_W(x_k, y_k)$ and $f_W(x_k)>f_W(x_k, y_k)$.
   By Lemma~\ref{lemma3-1},
 we have $f_W(x_k)>f_W(y_k)$. Further by
Corollary~\ref{corollary3.4}, we have $f_{X_{\ge k+1}}(x_{k+1})\ge
f_{Y_{\ge k+1}}(y_{k+1}).$ Since $f_{X_{\ge
k+1}}(x_{k+1})=f_{X_{k+1}}(x_{k+1})(1+f_{X_{\ge k+2}}(x_{k+2}))$ and
$f_{Y_{\ge k+1}}(y_{k+1})=f_{Y_{k+1}}(y_{k+1})(1+f_{Y_{\ge
k+2}}(y_{k+2}))$, we have $f_{X_{\ge k+2}}(x_{k+2})\ge f_{Y_{\ge
k+2}}(y_{k+2})$ since we assumed $f_{X_{k+1}}(x_{k+1})< f_{Y_{k+1}}(y_{k+1})$. Therefore,
 $\varphi(T^{\prime})>\varphi(T)>0$, contradicting to
 the optimality of $T$. So the
assertion holds.
\end{proof}

\begin{lemma}\label{lemma3-6}
Let $P$ be a path of an optimal $T$ in $\mathcal{T}_{\pi}$ whose end
vertices are leaves.

\noindent {\bf (i)} If the length of $P$ is odd ($2m-1$), then the vertices of $P$
can be labeled as $x_mx_{m-1}\cdots x_1$ $y_1y_2\cdots y_m$ such
that
$$f_{X_1}(x_1)\ge f_{Y_1}(y_1)\ge f_{X_2}(x_2)\ge f_{Y_2}(y_2)\ge\cdots
\ge f_{X_m}(x_m)= f_{Y_m}(y_m)=1.$$

\noindent {\bf (ii)} If the length of $P$ is even ($2m$), then the vertices of $P$ can
be labeled as $x_{m+1}x_mx_{m-1}\cdots x_1$ $y_1y_2\cdots y_m$ such
that
$$f_{X_1}(x_1)\ge f_{Y_1}(y_1)\ge f_{X_2}(x_2)\ge f_{Y_2}(y_2)\ge
\cdots\ge  f_{X_m}(x_m)\ge f_{Y_m}(y_m)= f_{X_{m+1}}(x_{m+1})=1.$$
\end{lemma}
\begin{proof}
We provide the proof of part (i), part (ii) can be shown in a similar manner.

Obviously, the vertices of $P$ may be labeled as $x_rx_{r-1}\cdots
x_1y_1y_2\cdots y_s$ such that $f_{X_1}$ the maximum among  $f_{X_i}$ and $f_{X_j}$ for $i=1,2,\cdots,r$
and $j=1,2,\cdots,s$, where $r+s=2m$. Therefore, there is only one of the following three cases:
%That is

%Without loss of generality, there exists a $1\le k\le m$ such
%that the vertices of $P$ may be labeled as $x_rx_{r-1}\cdots
%x_1y_1y_2\cdots y_s$  and

\noindent {\bf Case 1:} If the number of the maximum components is one, then there exists a $1\le k\le m$ such that
\begin{equation}\label{equ3-1}
f_{X_1}(x_1)>f_{Y_1}(y_1),\  \ \ f_{Y_1}(y_1)=
f_{X_{2}}(x_{2}),\cdots, f_{Y_{k-1}}(y_{k-1})= f_{X_{k}}(x_{k}),
f_{Y_k}(y_k)>f_{X_{k+1}}(x_{k+1})
\end{equation}
%解释：这是说明取得最大的子树数分支只有一个的情形。为什么说一定存在$1\le k\le m$使得（1）后面的式子成立
Next we will prove (1). It is divided into three subcases.

   \noindent {\bf Case 1.1:} If $f_{Y_1}(y_1)>f_{X_2}(x_2)$, then we have $k=1$ and (1)holds.

   \noindent {\bf Case 1.2:} If $f_{Y_1}(y_1)<f_{X_2}(x_2)$, then the vertices of $P$ may be relabeled such that$y_{i}$ is instead by $x_{i+1}$ for $i=1,\cdots,s$ and $X_{i}$ is instead by $y_{i-1}$ for $i=2,\cdots,r$. Hence it is  the same as the subcase 1.1.

   \noindent {\bf Case 1.3:} If $f_{Y_1}(y_1)=f_{X_2}(x_2)$.  Then we must have
                $f_{Y_2}(y_2)>f_{X_3}(x_3)$ or $f_{Y_2}(y_2)<f_{X_3}(x_3)$ or $f_{Y_2}(y_2)=f_{X_3}(x_3)$.

           \noindent {\bf Case 1.3.1:} If $f_{Y_2}(y_2)>f_{X_3}(x_3)$, then we have $k=2$ and (1)holds.

           \noindent {\bf Case 1.3.2:} If $f_{Y_2}(y_2)<f_{X_3}(x_3)$, then the vertices of $P$ may
                                be relabeled such that$y_{i}$ is instead by $x_{i+1}$ for $i=1,\cdots,s$
                                and $X_{i}$ is instead by $y_{i-1}$ for $i=2,\cdots,r$.
                                Hance, the case is the same as the subcase 1.3.1.

           \noindent {\bf Case 1.3.3:} If $f_{Y_2}(y_2)=f_{X_3}(x_3)$, we can continue to analyze like
                                $f_{Y_1}(y_1)=f_{X_2}(x_2)$. Then we have $k\ge 3$ and (1)holds.
Next we will prove that if (\ref{equ3-1}) holds, then we must have

$$r=s=m.$$
Otherwise, if $r<s$,  then by Lemma~\ref{lemma3.4}, $f_{X_i}(x_i)\ge
f_{Y_i}(y_i)$ for $i=1, \cdots, r$. Hence by
Corollary~\ref{corollary3.4}  we have $f_{X_{\ge r}}(x_r)\ge
f_{Y_{\ge r}}(y_r).$ On the other hand,  it is clear that $f_{X_{\ge r}}(x_r)=1$ and
$f_{Y_{\ge r}}(y_r)\ge 2$, contradiction.

If $r>s$, then $r\ge s+2$ since $r+s=2m$. Now we consider the path from
vertex $x_{s+1}$ to $y_s$. By Lemma~\ref{lemma3.4}, we have
$f_{Y_i}(y_i)\ge f_{X_{i+1}}(x_{i+1})$ for $i=1, \cdots, s$.
Further, by Corollary~\ref{corollary3.4}, we have $f_{Y_{\ge
s}}(y_s)\ge f_{X_{\ge s+1}}(x_{s+1}).$ Similarly, since $f_{Y_{\ge
s}}(y_s)=1$ and $f_{X_{\ge s+1}}(x_{s+1})\ge 2$,
contradiction. Therefore $r=s=m$.

Now by
Lemma~\ref{lemma3.4} applied to the path from $x_m$ to $y_m$, we have $f_{X_{
i}}(x_i)\ge f_{Y_{ i}}(y_i)$ for $i=1, \cdots, m$. On the other
hand, by Lemma~\ref{lemma3.4} applied to the path  from $y_{m-1}$ to $x_m$,
we have $f_{Y_{ i}}(y_i)\ge f_{X_{ i+1}}(x_{i+1})$ for $i=1,2,\dots
,m-1$. Hence the assertion holds.\\
%下面解释：取得最大的子树数分支有偶数个的情形
\noindent {\bf Case 2:} If the number of the maximum components is $2k\ge 2$. Then the path $P$ can be labeled as $x_mx_{m-1}\cdots x_1y_1y_2\cdots y_m$ such that
\begin{equation}\label{equ3-2}
f_{X_1}(x_1)=f_{Y_1}(y_1)=\cdots  =f_{X_k}(x_k)= f_{Y_k}(y_k)>
f_{X_{k+1}}(x_{k+1})\ge f_{Y_{k+1}}(y_{k+1}),
\end{equation}%解释：关键点说明最大分支在路上对应的点一定是相邻的
and the vertices $x_1,x_2,\cdots, x_k,y_1,y_2,\cdots, y_k$ are in the maximum components respectively.
That is to say all the maximum components are adjoining. Otherwise, there must be two pair vertices satisfying the first inequality in (1). Hence either of them, the vertices of $P$ may be labeled as
$x_{r_1}x_{r_1-1}\cdots x_1y_1y_2\cdots y_{s_1}$ or $x_{r_2}x_{r_1-1}\cdots x_1y_1y_2\cdots y_{s_2}$. By the case 1, we can have $r_1=r_2=s_1=s_2=m$. But it is impossible.Therefore, if there are more than one component with the most subtrees containing the vertex on the path $P$, then all of them must adjoin.
%下面解释：取得最大的子树数分支有奇数个且多于一个的情形

\noindent {\bf Case 3:} If the number of the maximum components is $2k+1>2$. Then the path $P$ can be labeled as $x_mx_{m-1}\cdots x_1y_1y_2\cdots y_m$ such that
\begin{equation}\label{equ3-3}
f_{X_1}(x_1)=f_{Y_1}(y_1)=\cdots=  f_{X_{k}}(x_{k})=
f_{Y_{k}}(y_{k})= f_{X_{k+1}}(x_{k+1})> f_{Y_{k+1}}(y_{k+1}),
\end{equation}
We omits the details.

Then Cases (\ref{equ3-2}) or (\ref{equ3-3}) can be handled in the same manner, we omit the details here.
\end{proof}

Following the conditions in Lemma~\ref{lemma3-6}, we have the following.

\begin{lemma}\label{lemma3-9}
\noindent {\bf (i)} If case (i) of Lemma~\ref{lemma3-6} holds, then

$$f_{T}(x_{1})\geq f_{T}(y_{1})> f_{T}(x_{2})\geq f_{T}(y_{2})>
\cdots > f_{T}(x_{m})\geq f_{T}(y_{m}).$$ Moreover, if $
f_T(x_k)=f_T(y_k)$ for some $1\le k\le m,$ then
 $f_T(x_i)=f_T(y_i)$ for $i=k, \cdots, m$.

\noindent {\bf (ii)} If case (ii) of Lemma~\ref{lemma3-6} holds, then
$$f_{T}(x_{1})> f_{T}(y_{1})\geq f_{T}(x_{2})> f_{T}(y_{2})\geq
\cdots \geq f_{T}(x_{m})>f_{T}(y_{m})\geq f_{T}(x_{m+1}).$$
Moreover, if $f_T(y_k)=f_T(x_{k+1})$ for some $1\le k\le m$, then
$f_T(y_i)=f_T(x_{i+1})$ for $i=k, \cdots m$.
\end{lemma}
\begin{proof}
 We only prove part (i), part (ii) is similar.

For any $2\le k\le m$,
 let $W_{k-1}$ be the connected component of $T$ containing vertices $x_{k-1}$ and
 $y_{k-1}$ after removing the
 edges $x_{k-1}x_k$ and $y_{k-1}y_k$.
For $k=1$ and $k=m$, it is easy to see
$$f_T(x_1)-f_T(y_1)=f_{X_1}(x_1)(1+f_{X\ge 2}(x_2))-f_{Y_1}(y_1)(1+f_{Y\ge
 2}(y_2))$$
 and
 $$f_T(x_m)-f_T(y_m)=f_{X_m}(x_m)(1+f_{W_{m-1}}(x_{m-1}))-f_{Y_m}(y_m)(1+f_{W_{m-1}}(y_{m-1})).
$$
  Moreover,
\begin{equation}\label{lemma3-9-1}
f_T(x_k)=f_{X_k}(x_k)(1+f_{X_{\ge
k+1}}(x_{k+1}))(1+f_{W_{k-1}}(x_{k-1})+
 f_{W_{k-1}}(x_{k-1}, \cdots, y_{k-1})f_{Y_k}(y_k)(1+f_{Y_{\ge
 k+1}}(y_{k+1})))
 \end{equation}
 and
 \begin{equation}\label{lemma3-9-2}
f_T(y_k)=f_{Y_k}(y_k)(1+f_{Y_{\ge
k+1}}(y_{k+1}))(1+f_{W_{k-1}}(y_{k-1})+
 f_{W_{k-1}}(y_{k-1}, \cdots, x_{k-1})f_{X_k}(x_k)(1+f_{X_{\ge
 k+1}}(x_{k+1}))).
 \end{equation}
By equations (\ref{lemma3-9-1}) and (\ref{lemma3-9-2}), we
have
\begin{eqnarray}\label{lemma3-9-3}
f_T(x_k)-f_T(y_k)&=&f_{X_k}(x_k)(1+f_{W_{k-1}}(x_{k-1}))(1+f_{X_{\ge
k+1}}(x_{k+1}))\nonumber \\
&&-f_{Y_k}(y_k)(1+f_{W_{k-1}}(y_{k-1}))(1+f_{Y_{\ge k+1}}(y_{k+1})).
\end{eqnarray}

Now we claim that for $1\le k\le m-1$,
\begin{equation}\label{lemma3-9-4}
f_{X_{\ge k+1}}(x_{k+1})\ge f_{Y_{\ge k+1}}(y_{k+1}),
\end{equation}
If  there is at least one strict inequality  in $f_{X_i}(x_i)\ge
f_{Y_i}(y_i)$ for $i=1, \cdots, k$, then by Lemma~\ref{lemma3.4},
 (\ref{lemma3-9-4}) holds.

If $f_{X_i}(x_i)= f_{Y_i}(y_i)$ for $i=1,
\cdots, k$ and there exists a $k<l<m$ such that $f_{X_i}(x_i)=
f_{Y_i}(y_i)$ for $i=1, \cdots, l-1$ and $f_{X_l}(x_l)>
f_{Y_l}(y_l)$.  Then by Lemma~\ref{lemma3.4}, we have $f_{X_{\ge
l+1}}(x_{l+1})\ge f_{Y_{\ge l+1}}(y_{l+1})$. Moreover,
\begin{equation}\label{lemma3-9-5}
f_{X_{\ge
k+1}}(x_{k+1})=\sum_{j=k+1}^l\prod_{i=k+1}^jf_{X_{i}}(x_{i})+
f_{X_{\ge l+1}}(x_{l+1})\prod_{i=k+1}^lf_{X_{i}}(x_{i})
\end{equation}
and
\begin{equation}\label{lemma3-9-6}
f_{Y_{\ge
k+1}}(y_{k+1})=\sum_{j=k+1}^l\prod_{i=k+1}^jf_{Y_{i}}(y_{i})+
f_{Y_{\ge l+1}}(y_{l+1})\prod_{i=k+1}^lf_{Y_{i}}(y_{i}).
\end{equation}
By  equations (\ref{lemma3-9-5}) and (\ref{lemma3-9-6}), the
claim holds.

If  $f_{X_i}(x_i)= f_{Y_i}(y_i)$ for $i=1, \cdots, m$,
then by equations (\ref{lemma3-9-5}) and (\ref{lemma3-9-6}), we have
$f_{X_{\ge k+1}}(x_{k+1})= f_{Y_{\ge k+1}}(y_{k+1})$ and the claim
holds.

Hence (\ref{lemma3-9-4}) is proved.

On the other hand,  by Lemma~\ref{lemma3-1}, we have
$f_{W_{k-1}}(x_{k-1})\ge f_{W_{k-1}}(y_{k-1})$.  Together
with (\ref{lemma3-9-4}), we see that $(\ref{lemma3-9-3})\ge 0$. Then $f_T(x_k)\ge f_T(y_k)$.

Now we prove $f_T(y_k)\ge f_T(x_{k+1})$ for any $1\le k\le m-1$. Let
$U_k$ be the connected component of $T$ containing
vertex $x_k$ after removing the edges
$y_{k-1}y_k$ (if $k=1$, let $y_0=x_1$) and $x_kx_{k+1}$. Then
\begin{equation}\label{lemma3-9-7}
f_T(y_k)=f_{Y_k}(y_k)(1+f_{Y_{\ge
k+1}}(y_{k+1}))(1+f_{U_k}(y_{k-1})+
 f_{U_k}(y_{k-1}, \cdots, x_{k})f_{X_{k+1}}(x_{k+1})(1+f_{X_{\ge
 k+2}}(y_{k+2})))
\end{equation}
and
\begin{equation}\label{lemma3-9-8}
f_T(x_{k+1})=f_{X_{k+1}}(x_{k+1})(1+f_{X_{\ge
k+2}}(x_{k+2}))(1+f_{U_k}(x_{k})+
 f_{U_k}(x_{k}, \cdots, y_{k-1})f_{Y_{k}}(y_{k})(1+f_{Y_{\ge
 k+1}}(y_{k+1}))).
\end{equation}
Similar to (\ref{lemma3-9-4}), we can show that $f_{Y_{\ge
 k+1}}(y_{k+1})\ge f_{X_{\ge
 k+2}}(x_{k+2})$. By Lemma~\ref{lemma3-1}, we have $f_{U_k}(y_{k-1})\ge
 f_{U_k}(x_{k})$. Hence (\ref{lemma3-9-7}) and
 (\ref{lemma3-9-8}) imply that
\begin{eqnarray}\label{lemma3-9-8-1}
f_T(y_k)-f_T(x_{k+1})&=& f_{Y_k}(y_k)(1+
f_{U_k}(y_{k-1}))(1+f_{Y_{\ge
k+1}}(y_{k+1})) \nonumber \\
&&-f_{X_{k+1}}(x_{k+1})(1+f_{U_k}(x_{k}))(1+f_{X_{\ge k+2}}(x_{k+2})) \nonumber \\
&=& f_{Y_{\ge k}}(y_{k})(1+ f_{U_k}(y_{k-1}))-f_{X_{\ge
k+1}}(x_{k+1})(1+f_{U_k}(x_{k}))\ge  0.
\end{eqnarray}
Moreover, if $f_T(x_k)=f_T(y_k)$ for some $1\le k\le m$, then by
(\ref{lemma3-9-3}), we have

 \begin{equation}\label{lemma3-9-9}
 f_{X_k}(x_k)=f_{Y_k}(y_k),\ \ f_{X_{\ge k}}(x_k)=f_{Y_{\ge
 k}}(y_k),\ \
 f_{W_{k-1}}(x_{k-1})=f_{W_{k-1}}(y_{k-1}).
 \end{equation}
 Since $f_{X_{\ge k}}=f_{X_k}(x_k)(1+f_{X_{\ge k+1}}(x_{k+1}))$ and
$f_{Y_{\ge k}}=f_{Y_k}(y_k)(1+f_{Y_{\ge k+1}}(y_{k+1}))$, we have $
f_{X_{\ge k+1}}(x_{k+1})=f_{Y_{\ge
 k+1}}(y_{k+1})$ by (\ref{lemma3-9-9}).
 On the other hand, since
$$f_{W_k}(x_k)=f_{X_k}(x_k)(1+f_{W_{k-1}}(x_{k-1})+f_{W_{k-1}}(x_{k-1},
\cdots, y_{k-1})f_{Y_k}(y_k))$$ and
$$f_{W_k}(y_k)=f_{Y_k}(y_k)(1+f_{W_{k-1}}(y_{k-1})+f_{W_{k-1}}(y_{k-1},
\cdots, x_{k-1})f_{X_k}(x_k)),$$ we have $f_{W_k}(x_k)=f_{W_k}(y_k)$
by (\ref{lemma3-9-9}). Hence
\begin{eqnarray}\label{lemma3-9-10}
f_T(x_{k+1})&=& f_{X_{\ge k+1}}(x_{k+1})(1+f_{W_{k}}(x_k)+f_{W_{k}}(x_{k},
\cdots, y_{k})f_{Y_{\ge k+1}}(y_{k+1}))  \nonumber \\
 &=& f_{Y_{\ge k+1}}(y_{k+1})(1+f_{W_{k}}(y_k)+f_{W_{k}}(x_{k},
\cdots, y_{k})f_{X_{\ge k+1}}(x_{k+1}))=f_T(y_{k+1}).
\end{eqnarray}
Therefore we have $f_T(x_{i})=f_T(y_{i})$ for $i=k, \cdots, m$.

Finally, we prove that $f_T(y_i)>f_T(x_{i+1})$ for $i=1, \cdots,
m-1$.  Suppose that $f_T(y_k)=f_T(x_{k+1})$ for some $1\le k\le m$.
Then by equation (\ref{lemma3-9-8-1}), we have $
f_{Y_k}(y_k)=f_{X_{k+1}}(x_{k+1})$ and  $f_{Y_{\ge
k}}(y_k)=f_{X_{\ge k+1}}(x_{k+1})$.
 Moreover,
$$f_{Y_k}(y_k)(1+f_{Y_{\ge k+1}}(y_{k+1}))=f_{Y_{\ge
k}}(y_k)=f_{X_{\ge k+1}}(x_{k+1})=f_{X_{k+1}}(x_{k+1})(1+f_{X_{\ge
k+2}}(x_{k+2})).$$
 Hence $f_{Y_{\ge
k+1}}(y_{k+1})=f_{X_{\ge k+2}}(x_{k+2})$.  Continuing this way in an inductive manner, we
have $f_{Y_{\ge m-1}}(y_{m-1})=f_{X_{\ge m}}(x_m)$. But
$f_{Y_{\ge m-1}}(y_{m-1})\ge 2 $ and $f_{X_{\ge m}}(x_m)=1$,
contradiction.

Combining the above results, we have proved part (i).
\end{proof}

The next Lemma relates the number of subtrees to the structure of the tree.

\begin{lemma}\label{lemma3-7}
For a path
$P(x_m,y_m)=x_{m}x_{m-1}\dots x_{2}x_{1}(z)y_{1}y_{2}\dots
y_{m-1}y_m $ in an optimal tree $T$, if $f_{X_{i}}(x_i)\ge
f_{Y_{i}}(y_{i})$ for $i=1, \cdots, k$,
%with at lease one strict inequality,
 $1\le k\le m-1$, then $d(x_{k})\geq d(y_{k})$.

Moreover, if  $f_{X_{i}}(x_i)= f_{Y_{i}}(y_{i})$ for $i=1, \cdots,
k$,
%with at lease one strict inequality,
 $1\le k\le m-1$, then $d(x_k)=d(y_k)$.\end{lemma}
\begin{proof} Suppose that $d(x_{k})< d(y_{k})$, let $r=d(y_{k})-d(x_{k})\ge 1$ and
$y_{k}u_i\in Y_{\ge k}$ for $i=1, \cdots, r$.

Further let $W$ be
the connected component of $T$ containing vertices $x_{k}$ and $y_{k}$ after
 removing the $r$ edges $y_{k}u_1,
\cdots, y_{k}u_r$. Let $X$
be the single vertex $x_{k}$ and let $Y$ be the connected component of
$T$ containing vertex $y_{k}$ after removing all edges incident to $y_k$
 except for the $r$ edges
$y_{k}u_1, \cdots, y_{k}u_r$.
Since $f_{X_i}(x_i)\ge f_{Y_i}(y_i)$ for $i=1, \cdots, k$, it is easy to
see that $f_W(x_{k})>f_W(y_{k})$ and $f_X(x_{k})=1< 2\le
f_{Y}(y_{k})$. By Lemma~\ref{lemma3-2}, there exists another
tree $T^{\prime}\in \mathcal{T}_{\pi}$ such that
$\varphi(T)<\varphi(T^{\prime})$, contradicting to the optimality of $T$.

Therefore the assertion holds. The case of equality is similar.
\end{proof}

From
Lemmas~\ref{lemma3-6}, \ref{lemma3-9} and \ref{lemma3-7} we have the following Lemma that
decides the `center' of the optimal tree.

\begin{lemma}\label{lemma3-10}
Let $T$ be an optimal tree  in $\mathcal{T}_{\pi}$ . If $f_T
(v_{0})=\max\{f_T (v), v\in V(T)\}$, then $d(v_0)=\max\{d(v), v\in
V(T)\}$.
\end{lemma}
\begin{proof}  The assertion clearly holds for small trees, so we
assume that $|V(T)|\ge 4$. Suppose that $d(v_0)<\max\{d(v), v\in V(T)\}$.
Then there exists a vertex $w$ such that $d(v_0)<d(w)$.  By
Theorem~9.1 in \cite{Szelely2005}, $f_T(v)$ is maximized at one or two adjacent vertices of $T$. Thus we have $f_T(v_0)>f_T(v)$ for
$v\in V(T)\setminus\{v_0\}$,  or  $f_T(v_0)=f_T(v_1)>f_T(v)$ for
$v\in V(T)\setminus\{v_0, v_1\}$ and $v_0v_1 \in E(T)$.

\noindent {\bf Case 1:} $f_T(v_0)>f_T(v)$ for $v\in V(T)\setminus\{v_0\}$. Hence,
 $f_T(v_0)>f_T(w)$. It is
 easy to see that $v_0$ is not a leaf (otherwise, let $u$ be a neighbor of $v_0$
 and we have
 $f_T(u)>f_T(v_0)$).
Let $P$ be a path containing vertex $v_0$ and $w$ whose end vertices
are leaves. Let the length of $P$ be $2m-1$ (the even length case is similar). Then by
Lemma~\ref{lemma3-6}, the vertices of $P$ can be labeled as
$P=x_m\cdots x_1y_1\cdots y_m$ such that $$f_{X_1}(x_1)\ge
f_{Y_1}(y_1)\ge f_{X_2}(x_2)\ge f_{Y_2}(y_2)\ge \cdots \ge
f_{X_m}(x_m)=f_{Y_m}(y_m)=1.$$ Hence by Lemma~\ref{lemma3-9}, we
have
$$f_{T}(x_{1})\geq f_{T}(y_{1})\geq f_{T}(x_{2})\geq f_{T}(y_{2})\geq
\dots \geq f_{T}(x_{m})\ge f_{T}(y_{m}).$$ Therefore $x_1$ must be
$v_0$ and $w$ must be $x_k$ for $2\le k\le m$ or $y_j$  for $1\le
j\le m$. By Lemma~\ref{lemma3-7}, we have $d(v_0)=d(x_1)\ge
d(x_k)=d(w)$ or $d(v_0)=d(x_1)\ge d(y_j)=d(w)$,
contradiction. Hence the assertion holds.

\noindent {\bf Case 2:} $f_T(v_0)=f_T(v_1)>f_T(v)$ for $v\in
V(T)\setminus\{v_0, v_1\}$ and $v_0v_1 \in E(T)$. If $w= v_1$,
then by Lemma~\ref{lemma3-7}, we have $d(w)=d(v_1)=d(v_0)<d(w)$, contradiction.

Hence we assume that $w\neq v_1$. First note that $v_0$ and
$v_1$ are not leaves. Let $P$ be a path containing
vertices $v_0$, $v_1$ and $w$ whose end vertices are leaves. Let the length
 of $P$ be $2m-1$ (the even case is similar), then by
Lemma~\ref{lemma3-6}, the vertices of $P$ can be labeled as
$P=x_m\cdots x_1y_1\cdots y_m$ such that
$$f_{X_1}(x_1)\ge f_{Y_1}(y_1)\ge f_{X_2}(x_2)\ge f_{Y_2}(y_2)\ge \cdots
\ge f_{X_m}(x_m)\ge f_{Y_m}(y_m)=1.$$ Hence by Lemma~\ref{lemma3-9},
we have
$$f_{T}(x_{1})\geq f_{T}(y_{1})\geq f_{T}(x_{2})\geq f_{T}(y_{2})\geq
\dots \geq f_{T}(x_{m})\ge f_{T}(y_{m}).$$ Therefore $\{x_1,
y_1\}=\{v_0, v_1\}$ and $w$ must be $x_k$ or $y_k$ for $1<k\le m$.
By Lemma~\ref{lemma3-7}, $d(v_0)\ge d(w)$ and $d(v_1)\ge d(w)$,
contradiction.

Combining cases (1) and (2), the assertion is proved.
\end{proof}

\begin{lemma}\label{lemma3-11} Let $T$ be an optimal tree in
$\mathcal{T}_{\pi}$.  If there is a path  $P=u_lu_{l-1}\cdots u_1
v_0v_1\cdots v_k$ with $f_T(v_0)=\max\{f_T(v):\ \  v\in V(P)\}$,
$f_{T}(u_1)\ge f_T({v_1})$, and $l=k$ (or \ $l= k+1$),  then
$$f_T(u_1)\ge f_T(v_1)\ge f_{T}(u_2)\ge\cdots \ge f_T({u_k})\ge f_T(v_k)
\hbox{ }({\rm or} \ \ \ge f_T(u_{k+1}))$$
and
$$d(u_1)\ge d(v_1)\ge d(u_2)\ge\cdots\ge d(u_k)\ge d(v_k) \hbox{ }({\rm or} \ \ \ge
  d(u_{k+1})).$$
\end{lemma}

\begin{proof}
Clearly, there exists  a path $Q$ that contains the path $P$
and its end vertices are leaves. We assume  $l=k$ (the $l=k+1$ case is similar).

Let the length of $Q$ be $2m-1$ (the even length case is similar). By
Lemmas~\ref{lemma3-9} and \ref{lemma3-7}, The vertices of $Q$ can be labeled as
$Q=x_{m}x_{m-1}\cdots x_1y_1\cdots y_m$ such that
$$f_{T}(x_{1})\geq
f_{T}(y_{1})> f_{T}(x_{2})\ge f_{T}(y_{2})> \dots >
f_{T}(x_{m})\ge f_{T}(y_{m}) $$and
$$
d(x_1)\ge d(y_1)\ge d(x_2)\ge d(y_2)\ge \cdots \ge
d(x_m)=d(y_m)=1.$$

{\bf Case 1:} $v_0=x_1$. We must have $u_1=y_1$ and $v_1=x_2$. Then $u_i=y_i$ and
$v_i=x_{i+1}$ for $i=1, \cdots, k$. Hence the assertion holds.

{\bf Case 2:} $v_0=x_i$ for $i>1$. Then $f_T(v_0)\ge
f_T(x_1)\ge f_T(y_1)\ge f_T(x_i)=f_T(v_0)$, which implies
$f_T(x_1)=f_T(y_1)=f_T(v_0)$ and contradicts to Theorem~9.1 in
\cite{Szelely2005}.

{\bf Case 3:} $v_0=y_i$. Then $i=1$ and
$f_T(x_1)=f_T(y_1)=f_T(v_0)$.  We must have $u_1=x_1$ and $ v_1=y_2$. Then $u_i=x_i$ and $v_i=y_{i+1}$
for $i=1, \cdots, k$. So the assertion holds.
\end{proof}

Now for an optimal tree $T$ in
$\mathcal{T}_{\pi}$, let $v_0\in V(T)$ be the root of $T$  with
$f_T(v_0)=\max\{f_T(v):\ \  v\in V(T)\}$ and $d(v_0)=\max\{d(v):\ \
v\in V(T)\}$.

\begin{corollary}\label{corollary3-12}
If there is  a path  $ P=u_k\cdots u_1wv_1v_2\cdots v_k$ with
$dist(u_k, v_0)=dist(v_k,v_0)=dist(w, v_0)+k$ and $f_T(u_1)\ge
f_T(v_1)$, then
$$f_T(u_1)\ge f_T(v_1)\ge f_{T}(u_2)\ge\cdots \ge f_T({u_k})\ge f_T(v_k)$$
and
$$d(u_1)\ge d(v_1)\ge d(u_2)\ge\cdots\ge d(u_k)\ge d(v_k).$$

If there is  a path  $ P=u_{k+1}\cdots u_1wv_1v_2\cdots v_k$
with $dist(u_{k+1}, v_0)=dist(v_k,v_0)+1=dist(w, v_0)+k+1$ and
$f_T(u_1)\ge f_T(v_1)$, then
$$f_T(u_1)\ge f_T(v_1)\ge f_{T}(u_2)\ge\cdots \ge f_T({u_k})\ge f_T(v_k)\ge f_T(u_{k+1})$$
and
$$d(u_1)\ge d(v_1)\ge d(u_2)\ge\cdots\ge d(u_k)\ge d(v_k)\ge d(u_{k+1}).$$

\end{corollary}
\begin{proof}
If $w=v_0$, then the assertion follows from Lemma~\ref{lemma3-11}.
If $w\neq v_0$, then there exists  a path $Q$ containing vertices
$u_k, \cdots, u_1, w, v_0$ whose end vertices are leaves.
By Lemma~\ref{lemma3-11}, we have $f_T(w)\ge f_T(u_1)\ge\cdots \ge
f_T(u_k)$. Similarly, there exists a path $R$ containing vertices
$v_k, \cdots, v_1, w, v_0$ whose end vertices are leaves and we have $f_T(w)\ge f_T(v_1)\ge\cdots \ge
f_T(v_k)$. Therefore $f_T(w)=\max\{f_T(v): \ v\in V(P)\}$, the
assertion follows from Lemma~\ref{lemma3-11}.
\end{proof}

\section{Proofs of Theorems~\ref{theorem2-1} and \ref{theorem2-2}}

 Now we are ready to prove Theorems~\ref{theorem2-1} and \ref{theorem2-2}.

\begin{proof} of Theorem~\ref{theorem2-1}. Let  $T$ be an optimal tree in ${\mathcal {T}}_{
\pi}$.  By Lemma~\ref{lemma3-10}, there exists a vertex $v_0$ such
that $f_T(v_0)=\max\{f_T(v):\ \ v\in V(T)\}$ and $d(v_0)=\max\{d(v):
\ \ v\in V(T)\}$. Let $v_0$ be the root of $T$ and put $V_i=\{v:\
dist(v, v_0)=i\}$ for $i=0, \cdots, p+1$  with
 $V(T)=\bigcup_{i=0}^{p+1} V_i$.  Denote by $|V_i|=s_i$ for $i=1, \cdots, p+1$.
  We now can relabel the vertices of
 $V(T)$ by the recursion method. For $V_0$, relabel $v_0$ by $v_{01}$ as the
 root of tree $T$.
  The vertices of $V_1$ (consisting of all
 neighbors $v_{01}$) are relabeled as
$v_{11},\cdots, v_{1,s_1}$,  satisfying:
 \begin{eqnarray*}f_T(v_{11})\ge f_T(v_{12})\ge\cdots \ge f_T(v_{1,s_1})\end{eqnarray*}
 and
 \begin{eqnarray*} f_T(v_{1i})=f_T(v_{1j})\ \  {\rm implies }\ \ d(v_{1i})\ge
 d(v_{1j}) \ {\rm for }\ 1\le i<j\le s_1.\end{eqnarray*}
% Moreover,
% $s_1=d(v_{01})$.
Generally, we assume that all vertices of $V_i$ are relabeled as
$\{v_{i1}, \cdots, v_{i,s_i}\}$ for $i=1,\cdots, t$. Now consider
all vertices in $V_{t+1}$. Since $T$ is tree, it is easy to see that $s_1 = d(v_{01})$ and
\begin{eqnarray*}s_{t+1}=|V_{t+1}|=d(v_{t1})+\cdots+d(v_{t, s_t})-s_t.\end{eqnarray*} Hence
for $1\le r\le s_t$, all neighbors in $V_{t+1}$ of $v_{tr}$ are
relabeled as \begin{eqnarray*}v_{t+1, d(v_{t1})+\cdots+d(v_{t,
r-1})-(r-1)+1}, \cdots, v_{t+1, d(v_{t1})+ \cdots+d(v_{t,
r})-r}\end{eqnarray*}
 and satisfy
the conditions: \begin{eqnarray}f_T(v_{t+1, i})\ge f_T(v_{t+1,
j})\end{eqnarray}
 and
 \begin{eqnarray}
  f_T(v_{t+1,i})=f_T(v_{t+1,j})\ \  {\rm implies }\ \ d(v_{t+1,i})\ge
 d(v_{t+1,j})\end{eqnarray}
 for $
d(v_{t1})+\cdots+d(v_{t, r-1})-(r-1)+1 \le i<j\le d(v_{t1})+
\cdots+d(v_{t, r})-r. $ In this way, we have relabeled all vertices
of $V(T)=\bigcup_{i=0}^{p+1}V_i$. Therefore, we are able to define a
well-ordering of vertices in $V(T)$ as follows: \begin{equation}
\label{wellordering}
 v_{ik}\prec v_{jl}, \ \ {\rm if}\  0\le i<
j\le p+1 \ \ {\rm or}\ \ i=j \ {\rm and }\ 1\le  k<l\le s_i.
\end{equation}

We need to prove that this well-ordering is a BFS-ordering of  $T$.
In other words, $T$ is isomorphic to $T_{\pi}^*$.

We first prove, for $t=0, \cdots, p+1$, the following
inequalities.
 \begin{equation}\label{mainine1}
f_T(v_{t1})\ge f_T(v_{t2})\ge \cdots\ge f_T(v_{t, s_t})\ge
f_T(v_{t+1,1})
\end{equation} and
\begin{equation}
\label{mainine2} d(v_{t1})\ge d(v_{t2})\ge \cdots \ge
d(v_{t,s_t})\ge d(v_{t+1,1}).
\end{equation}

For any two vertices $v_{ti}$ and $v_{tj}$ with $1\le i<j\le s_t$,
there exists a path $P=v_{ti}\cdots v_{k+1, l}$ $w_kv_{k+1,r}\cdots
v_{tj}$ with $l<r$, where $dist(v_{ti}, v_{01})=dist_(v_{tj},
v_{01})=dist(w_k,v_{01})+t-k$.  Then we have $f_T(v_{k+1,l})\ge
f_T(v_{k+1,r})$, $f_T(v_{ti})\ge f_T(v_{tj})$ and $d(v_{ti})\ge
d(v_{tj})$ by Corollary~\ref{corollary3-12}. On the other hand, we
consider the path $Q=v_{t+1,1}v_{t1}\cdots
v_{11}v_{01}v_{1s_1}\cdots v_{t,s_t}$. Then $f_T(v_{t, s_t})\ge
f_T(v_{t+1,1})$ and $d(v_{t,s_t})\ge d(v_{t+1,1})$ by
Corollary~\ref{corollary3-12}. Therefore (\ref{mainine1}) and
(\ref{mainine2}) hold for $t=0, \cdots, p+1$. That is
\begin{equation}\label{allordering} f_T(v_{01})\ge f_T(v_{11})\ge \cdots
\ge f_T(v_{1,s_1})\ge f_T(v_{21})\ge \cdots \ge
f_T(v_{2,s_2})\ge\cdots\ge f_T(v_{p+1, s_{p+1}})\end{equation}
 and
\begin{equation} \label{degeeorder}d(v_{01})\ge d(v_{11})\ge
d(v_{1,s_1})\ge d(v_{21})\ge \cdots\ge d(v_{2,s_2})\ge
d(v_{p+1,1})\ge d(v_{p+1, s_{p+1}}).
\end{equation}

By (\ref{wellordering}), (\ref{allordering}) and (\ref{degeeorder}),
it is easy to see that this well ordering satisfies all conditions
in Definition~\ref{def2-1}.
 Hence
 $T$ has a BFS-ordering. Further, by Proposition~2.2 in \cite{zhang2008}, $T$ is
isomorphic to $T_{\pi}^*$. So $T_{\pi}^*$ is the unique optimal tree
in ${\mathcal T}_{ \pi}$ having the largest number of subtrees.
\end{proof}

\begin{proof} of Theorem~\ref{theorem2-2}.
By proposition~\ref{prop}, without loss of generality, we assume
that $\pi=(d_0,d_1,\cdots,d_i,\cdots, d_j,\cdots ,d_{n-1})$ and
$\pi_1=(d_0,d_1,\cdots,d_i+1,\cdots, d_j-1,\cdots, d_{n-1})$ with
$i<j$, then we have $\pi\triangleleft \pi_1$. Let $T_{\pi}^*$ be the
optimal tree in $\mathcal{T}_{\pi}$. By the proof of
Theorem~\ref{theorem2-1}, the vertices of $T_{\pi}^*$ can be labeled
as the $V=\{v_0, \cdots, v_{n-1}\}$ such that
$$f_{T_{\pi}^*}(v_0)\ge  f_{T_{\pi}^*}(v_1)\ge\cdots\ge
f_{T_{\pi}^*}(v_{n-1})$$ and
$$d(v_0)\ge d(v_1)\ge\cdots\ge d(v_{n-1}),$$
where $d(v_l)=d_l$ for $l=0, \cdots, n-1$. Moreover, $v_0$ is the
root of $T_{\pi}^*$. There exists a vertex $v_k$ such that
$v_jv_k\in E(T_{\pi}^*)$ with $k>j$. Let $W$ be the tree achieved
from $T_{\pi}^*$ by removing the subtree induced by $v_k$. Moreover,
let $X$ be the single vertex $v_i$ and $Y$ be the subtree induced by $v_k$
with the edge $v_jv_k$ added, respectively.  Clearly,
$f_T(v_i)=f_W(v_i)+f_W(v_i, v_j)(f_Y(v_j)-1)$ and
$f_T(v_j)=f_W(v_j)+f_W(v_j)(f_Y(v_j)-1)$. Hence by $f_W(v_i,
v_j)<f_W(v_j)$ and $f_T(v_i)\ge f_T(v_j)$, we have
$f_W(v_i)>f_W(v_j)$. On the other hand, let $T_1$ be the tree from
$T$ by deleting the edge $v_jv_k$ and adding the edge $v_iv_k$. Then
the degree sequence of $T_1$ is $\pi_1$. By Lemma~\ref{lemma3-2}, we
have $ \varphi (T_{\pi}^*)<\varphi(T_1)$. Hence $\varphi
(T_{\pi}^*)<\varphi(T_1)\le \varphi(T_{\pi_1}^*).$

The assertion is then proved.
\end{proof}

\section{Applications of the Main Theorems}

In the end we use  Theorems~\ref{theorem2-1} and \ref{theorem2-2} to achieve
extremal graphs with the largest  number of subtrees  in some classes
of graphs. As corollaries, we provide proofs to some results in \cite{Szelely2005},
\cite{kirk2008}, etc.

Let ${\mathcal T}_{n,\Delta}^{(1)}$ be the
set of all trees of order $n$ with the largest degree $\Delta$,
${\mathcal T}_{n, s}^{(2)}$ be the set of all trees of order $n$
with $s$ leaves,   ${\mathcal T}_{n,\alpha}^{(3)}$ be the set of all
trees of order $n$ with the independence number $\alpha$ and
${\mathcal T}_{n,\beta}^{(4)}$ be the set of all trees of order $n$
with the matching number $\beta$.

\begin{corollary}(\cite{Szelely2005})
Let $T$ be any tree of order $n$. Then $$ \left(\begin{array}{c}
n+1\\ 2\end{array}\right) \le \varphi (T)\le 2^{n-1}+n-1$$ with left
equality if and only if $T$ is a path of order $n$ and the right
equality if and only if $T$ is the star $K_{1, n-1}.$
\end{corollary}

\begin{proof}
Let $T$ be a tree of order $n$ with degree sequence $\tau$. Let
$\pi_1=(2, \cdots, 2, 1, 1)$ and $\pi_2=(n-1, 1, \cdots, 1)$ with
$n$ terms. Clearly the path $P$ of order $n$ is the only tree with
the degree sequence $\pi_1$  and the star $K_{1, n-1}$ of order $n$
is the only tree with degree sequence $\pi_2$. Furthermore,
$\pi_1\triangleleft \tau \triangleleft \pi_2$. Hence by
Theorems~\ref{theorem2-1} and \ref{theorem2-2}, the assertion holds.
\end{proof}
\begin{corollary}\label{maxdegree}(\cite{kirk2008})
 There is only one optimal tree $T_{\Delta}^*$ in ${\mathcal
T}_{n,\Delta}^{(1)}$ with $\Delta\ge 3$,  where $T_{\Delta}^*$ is
$T_{\pi}^*$  with degree sequence $\pi$  as follows:
 Denote
$p=\lceil\log_{(\Delta-1)}\frac{n(\Delta-2)+2}{\Delta}\rceil-1$ and
$n-\frac{\Delta(\Delta-1)^p-2}{\Delta-2}=(\Delta-1)r+q$ for $0\le
q<\Delta-1$. If $q=0$, put $\pi=(\Delta,\cdots,\Delta,1, \cdots,1)$
with the number $\frac{\Delta(\Delta-1)^{p-1}-2}{\Delta-2}+r$ of
degree $\Delta$. If $ q\ge 1$, put $\pi=(\Delta,\cdots,\Delta,q,
1,\cdots,1)$ with the number
$\frac{\Delta(\Delta-1)^{p-1}-2}{\Delta-2}+r$ of degree $\Delta$.
\end{corollary}
\begin{proof}
For any tree $T$ of order $n$ with  the largest degree $\Delta$, let
$\pi_1=(d_0,\cdots, d_{n-1})$ be the nonincreasing degree sequence
of $T$. Assume that $T_{\Delta}^*$ has $p+2$ layers. Then there is a
vertex in layer 0 (i.e. the root), there are $\Delta$ vertices in layer
1, there are $\Delta(\Delta-1)$ vertices in layer 2, $\cdots$, there
are $\Delta(\Delta-1)^{p-1}$ vertices in layer $p$, there are at
most $\Delta(\Delta-1)^p$ vertices in layer $p+1$. Hence
\begin{eqnarray*}1+\Delta+\Delta(\Delta-1)+\cdots+\Delta(\Delta-1)^{p-1}< n\le
1+\Delta+\Delta(\Delta-1)+\cdots+\Delta(\Delta-1)^{p}.\end{eqnarray*}
  Thus \begin{eqnarray*}
\frac{\Delta(\Delta-1)^p-2}{\Delta-2}<n\le
\frac{\Delta(\Delta-1)^{p+1}-2}{\Delta-2}.\end{eqnarray*} Hence
\begin{eqnarray*}p=\lceil\log_{(\Delta-1)}\frac{n(\Delta-2)+2}{\Delta}\rceil-1\end{eqnarray*}
and there exist integers $r$ and $0\le q<\Delta-1$ such that
\begin{eqnarray*}n-\frac{\Delta(\Delta-1)^p-2}{\Delta-2}=(\Delta-1)r+q.\end{eqnarray*}
Therefore the degrees of all vertices from layer 0 to layer $p-1$
are $\Delta$ and there are $r$ vertices in layer $p$ with degree
$\Delta$.  Denote by
$m=\frac{\Delta(\Delta-1)^{p-1}-2}{\Delta-2}+r-1$. Then there are
$m+1$ vertices with degree $\Delta$ in $T_{\Delta}^*$. Hence  the
degree sequence of $T_{\Delta}^*\in {\mathcal T}_{n,\Delta}$  is
$\pi=(d_0^{\prime},\cdots, d_{n-1}^{\prime})$ with
$d_0^{\prime}=\cdots= d_m^{\prime}=\Delta$,
$d_{m+1}^{\prime}=\cdots=d_{n-1}^{\prime}=1$ for $q=0$; and is
$\pi%^{\prime}
=(d_0^{\prime},\cdots, d_{n-1}^{\prime})$ with
$d_0^{\prime}=\cdots= d_m^{\prime}=\Delta$, $d_{m+1}^{\prime}=q,$
$d_{m+2}^{\prime}=\cdots=d_{n-1}^{\prime}=1$ for $q=1$.
 It follows
from $d_i\le \Delta$ that $\sum_{i=0}^kd_i\le
\sum_{i=0}^kd_i^{\prime}$ for $k=0,\cdots, m$. Further by
$d_i^{\prime}=1\le d_i $ for $k=m+2,\cdots, n-1$,  we have
\begin{eqnarray*}
\sum_{i=0}^kd_i= 2(n-1)-\sum_{i=k+1}^{n-1}d_i\le
2(n-1)-\sum_{i=k+1}^{n-1}d_i^{\prime}=\sum_{i=0}^kd_i^{\prime}\end{eqnarray*}
 for $k=m+1,
\cdots n-1$. Thus $\pi_1\triangleleft\pi$. Hence by
Theorems~\ref{theorem2-1} and \ref{theorem2-2},  $\varphi(T)\le
\varphi(T_{\Delta}^*)$ with equality if and only if
$T=T_{\Delta}^*$.
\end{proof}

\noindent {\bf Remark} If $\Delta=3$ in Corollary~\ref{maxdegree}, then the
result is precisely Theorem 2.1 in \cite{Szelely2007}.

\begin{corollary}\label{leaves}
There is only one optimal  tree $T_s^*$ in  ${\mathcal T}_{n,
s}^{(2)}$  where $T_s^*$ is obtained from $t$ paths of order $q+2$ and
$s-t$ paths of order $q+1$ by identifying one end of the $s$
paths. Here $n-1=sq+t,
0\le t<s$. In other words, for any tree of order $n$ with $s$ leaves,
$$\varphi(T)\le %(s-t)(q+1)t(q+2)+q(q+1)(s-t)/2+(q+1)(q+2)t/2=
%\frac{q+1}{2}[sq+2t+2t(s-t)(q+2)]
(q+2)^t+(q+1)^{s-t+2}$$
with equality if and only if $T$
is $T_s^*$.
\end{corollary}
\begin{proof}
Let $T$ be any tree in ${\mathcal T}_{n, s}^{(2)}$ with the
nonincreasing degree sequence $\pi_1=(d_0,\cdots,d_{n-1})$. Thus
$d_{n-s-1}>1$ and $d_{n-s}=\cdots=d_{n-1}=1$. Let $T_{\pi}^*$ be a
BFS-tree with degree sequence $\pi=(s, 2,\cdots, 2,
1,\cdots, 1)$, where there are the number $s$ of 1$'$s in $\pi$. It is easy to see that
$\pi_1\triangleleft\pi$. By Theorem~\ref{theorem2-2}, the assertion
holds.
\end{proof}

\begin{corollary}\label{inde}
There is only one optimal tree $T_{\alpha}^*$ in  ${\mathcal
T}_{n,\alpha}^{(3)}$,  where $T_\alpha^*$  is
$T_{\pi}^*$ with degree sequence $\pi=(\alpha,
2,\cdots,2,1,\cdots,1)$ with numbers $n-\alpha-1$ of 2$'$s and  $\alpha$
of 1$'$s, i.e., $T_{\pi}^*$ is obtained from the star $K_{1, \alpha}$ by
adding $n-\alpha-1$ pendent edges to $n-\alpha-1$ leaves of $K_{1, \alpha}$. In other words, for any tree of order $n$
with the independence number $\alpha$,
$$\varphi(T)\le %2^{2\alpha-n+1}3^{n-\alpha-1}+2\alpha-n+1+3(n-\alpha-1)=
2^{2\alpha-n+1}3^{n-\alpha-1}+2n-\alpha-2$$ with equality if and
only if $T$ is $T_{\alpha}^*$.
\end{corollary}
\begin{proof}
For any tree $T$ of order $n$ with  the independence number
$\alpha$, let $I$ be an independent set of $T$ with size $\alpha$ and $\tau=(d_0,\cdots, d_{n-1})$ be the
degree sequence of $T$.  If there exists a  leaf $u$ with $u\notin I$, then there exists a vertex
$v\in I$  with $(u,v)\in E(T)$. Hence $I\bigcup \{u\}\setminus\{v\}$
is an independent set of $T$ with size $\alpha$. Therefore,
one can always construct an independent set of $T$ with size $\alpha$ that contains
all leaves of $T$. Hence there are at most $\alpha$
leaves.  Then $d_{n-\alpha-1}\ge 2$ and $\tau\triangleleft
\pi$. By Theorems~\ref{theorem2-1} and \ref{theorem2-2},
the assertion holds.
\end{proof}

\begin{corollary}\label{match}
There is only one optimal tree $T_{\beta}^*$ in ${\mathcal
T}_{n,\beta}^{(4)}$, where $T_{\beta}^*$ is $T_{\pi}^*$
 with degree sequence $\pi=(n-\beta, 2,\cdots,2,1,\cdots,1)$. Here the number of 1$'$s is $n-\beta$.
That is, $T_{\pi}^*$ is obtained from the
star $K_{1, n-\beta}$ by adding $\beta-1$ pendent edges to
$\beta-1$ leaves of $K_{1, n-\beta}$. In other words, for any
$T \in {\mathcal T}_{n,\beta}^{(4)}$,
$$\varphi(T)\le %2^{n-2\beta+1}3^{\beta-1}+3(\beta-1)+(n-2\beta+1)=
2^{n-2\beta+1}3^{\beta-1}+n-\beta-2$$ with equality if and only if
$T$ is $T_{\beta}^*$.
\end{corollary}
\begin{proof}
For any tree $T$ of order $n$ with matching number $\beta$, let
$\tau=(d_0,\cdots, d_{n-1})$ be the degree sequence of
$T$. Let $M$ be a matching of $T$ with size $\beta$.
Since $T$ is connected, there are at least $\beta$ vertices in $T$
such that their degrees are at least two. Hence $d_{\beta-1}\ge 2$
and $\tau\triangleleft \pi$. By Theorems~\ref{theorem2-1}
and \ref{theorem2-2}, the assertion holds.
\end{proof}

%\section{Acknowledgements} The authors would like to thank the
%anonymous referees very much for valuable suggestions, corrections
%and comments which results in a great improvement of the original
%manuscript.

\end{document}